\documentclass[12pt]{article}
\usepackage{amssymb,amsmath}
\usepackage{amsthm}
\usepackage{enumerate}
\usepackage{amsfonts}
\newcommand{\N}{{\mathbb N}}

\begin{document}
\newtheorem{thm}{Theorem}[section]
\newtheorem{defs}[thm]{Definition}
\newtheorem{lem}[thm]{Lemma}
\newtheorem{rem}[thm]{Remark}
\newtheorem{cor}[thm]{Corollary}
\newtheorem{prop}[thm]{Proposition}
\renewcommand{\theequation}{\arabic{section}.\arabic{equation}}
\newcommand{\newsection}[1]{\setcounter{equation}{0} \section{#1}}
%%%%%% title %%%%%%%%%%%%%%%%%%%%%%%%%%%%%%%
\title{Generalization of the theorems of Barndorff-Nielsen and Balakrishnan-Stepanov to Riesz spaces
      }
%%%%%%%%%%%%%%%%%%%%%%%%%%%%%%%%%%%%%%%
\author{Nyasha Mushambi, { Bruce A. Watson} \footnote{Supported in part by the Centre for Applicable Analysis and
Number Theory and by NRF grant number IFR2011032400120.}, Bertin Zinsou\\
School of Mathematics\\
University of the Witwatersrand\\
Private Bag 3, P O WITS 2050, South Africa}
\maketitle
%%%%%%%%%%%%% abstract %%%%%%%%%%%%%%%%%%%%%%
\abstract{\noindent
In a Dedekind complete Riesz space, $E$, we show that if $(P_n)$ is a sequence of band projections in $E$ then
$$\limsup\limits_{n\to \infty} P_n - \liminf\limits_{n\to \infty} P_n = \limsup\limits_{n\to \infty} P_n(I-P_{n+1}).$$
This identity is used to obtain conditional extensions in a Dedekind complete Riesz spaces with weak order unit and conditional expectation operator of the Barndorff-Nielsen and Balakrishnan-Stepanov generalizations of the First Borel-Cantelli Theorem.
}
%%%%%%%%%%%%%%%%%%%%%%%%%%%%%%%%%%%%%%%
\parindent=0in
\parskip=.2in
%%%%%%%%introduction %%%%%%%%%%%%%%%%%%%%%%%%%%
\newsection{Introduction}

In \cite[Lemma 1*]{bn} Barndorff-Nielsen gave the following generalization of the First Borel-Cantelli Theorem.

\begin{thm}[Barndorff-Nielsen]
Let $(\Omega, \mathcal{F}, P)$ be a probability space. If $( A_n)$ is a sequence of events in the probability space $(\Omega, \mathcal{F}, P)$ such that $\displaystyle{\lim_{n\to \infty} P(A_n)= 0}$ and
\begin{equation*}
 \sum\limits_{n=1}^\infty P(A_n \cap A^c_{n+1}) < \infty,
\end{equation*}
then
$\displaystyle{P( \limsup_{n \to \infty}  A_n) = 0}.$
\end{thm}

In  \cite[page 51]{chan}, it was shown that the assumption $\displaystyle{\lim_{n\to \infty}P(A_n)\to 0}$ could be weakened to $\displaystyle{\liminf_{n\to \infty} P(A_n)= 0}$.
In \cite[Lemma 3]{bs}, see also \cite[page 52]{chan}, Balakrishnan and Stepanov further extended the Barndorff-Nielsen Theorem, as below.

\begin{thm}[Balakrishnan-Stepanov]
Let $(\Omega, \mathcal{F}, P)$ be a probability space and $(A_n)$ be a sequence of events in $(\Omega, \mathcal{F}, P)$. If $P(A_n) \to 0$ as $n\to \infty$ and
\begin{equation*}
\sum\limits_{n=1}^\infty P(A^c_n\ \cap A^c_{n+1} \ \cap \ldots \cap \ A^c_{n+m-1} \ \cap \ A_{n+m}) < \infty 
\end{equation*} 
for some $m \geq 1$, then $P(A_n\ \textit{infinitely often}) = 0.$
\end{thm} 

We show, see Theorem \ref{l2}, that in a Dedekind complete Riesz space, $E$, if $(P_n)$ is a sequence of band projections, in $E$, then
$$\limsup\limits_{n\to \infty} P_n - \liminf\limits_{n\to \infty} P_n = \limsup\limits_{n\to \infty} P_n(I-P_{n+1}).$$

Let $E$ be a Dedekind Riesz space with weak order unit. A conditional expectation operator on $E$, as defined in \cite{wen1}, is 
a positive order continuous projection $T$ on $E$ with range, ${R}(T)$, a Dedekind complete Riesz subspace of $E$, and having $T(e)$ a weak order unit of $E$ for each weak order unit $e$ of $E$. We refer the reader to \cite{wen1} for background on conditional expectation operators on Riesz spaces and \cite{wen2} for stochastic processes in Riesz spaces with recent developments in \cite{azouzi} and \cite{grobler}. 

In the setting of a Dedekind complete Riesz spaces with weak order unit and conditional expectation operator we use the above
band projection identity to give conditional generalizations of the theorems Barndorff-Nielsen and Balakrishnan-Stepanov to Riesz spaces.

We thank the referees for their comments and corrections.
%%%%%%%%%%%%%%%%%%%%%%%%%%%%%%%%%%%%%%%%%%%%%%%%%%%%%%%
\newsection{Preliminaries}

We refer the reader to \cite{ali, lux, za1} for general background on Riesz spaces and band projections. 
From \cite{M-N}, or \cite[Theorem 12.7]{za1}, we have 
that if $(f_n)_{n\geq 1}$ is an order bounded sequence in the Dedekind complete Riesz space $E$, then the sequence converges if and only if
$\displaystyle{\liminf_{n\to \infty} f_n = \limsup_{n\to \infty} f_n}$ in which case this common value is the order limit of $(f_n)$ as 
$n\to \infty$ and is denoted $\displaystyle{\lim_{n\to \infty} f_n}$.

For a sequence $(P_n)$ of band projections in a Dedekind complete Riesz space, $E$, it is easy to verified that 
$\displaystyle{\prod_{n\in \mathbb{N}}P_n = \bigwedge_{n \in \mathbb{N}} P_n}$. 
Further for a sequence $(f_n)$ in the positive cone, $E_+$, of such a Riesz space $E$, if $(f_n)$ is order bounded and $f_n\wedge f_m = 0$ for all $n\neq m$ then $\displaystyle{\sum_{n\in \mathbb{N}}f_n = \bigvee_{n \in \mathbb{N}} f_n}$. This is the order limit of the result in \cite[page 110]{za1}.

The Fatou property (\ref{fatou}) holds for any positive order continuous operator, $T$, on a Dedekind complete Riesz space, however, our interest here is only with regards to conditional expectation operators. It is well known, but included here for completeness.

\begin{thm}[Fatou]\label{l1}
Let $E$ be a Dedekind complete Riesz space with conditional expectation operator $T$ with weak order unit $e= Te$. If $(f_n)$ is a  bounded sequence, in $E$, then
\begin{equation}\label{fatou}
T(\liminf\limits_{n\to \infty} f_n) \leq  \liminf\limits_{n\to \infty} Tf_n.
\end{equation}
\end{thm}

\begin{proof}
Let $f\in E$, denote the limit inferior of $f_n$. For every $k\in \mathbb{N}$, define
$\displaystyle{
g_k =  \inf\limits_{n\geq k} f_n}$.
Then the sequence $(g_k)$  is increasing and converges to $f$. For $k\leq n$, we have $g_k \leq f_n$ so, as $T$ is a positive linear operator,
$Tg_k \leq Tf_n$.
Now taking the infimum over $n\geq k$, we have
$\displaystyle{
Tg_k \leq \inf_{n\geq k} Tf_n}$.
We know $g_k \uparrow$ with order limit $f$, so by the order continuity of $T$ we have,
$\displaystyle{
\lim_{k\to \infty}Tg_k= Tf}$.
Also
$\displaystyle{
\inf_{n\geq k} Tf_n \uparrow \liminf_{n\to \infty} Tf_n}$
so taking $k\to \infty$, we have
$\displaystyle{
Tf \leq \liminf_{n\to \infty} Tf_n}$.\qedhere
\end{proof}

If $f_n \leq g$, for all $n\geq 1$, then applying Theorem \ref{l1} to $(g-f_n)$ gives following Reverse Fatou inequality.

\begin{cor}\label{tt3}
Let $E$ be a Dedekind complete Riesz space with conditional expectation operator $T$ with weak order unit $e= Te$. If  $(f_n)\subset E^+$ and $g\in E^+$ with $f_n \leq g$ for all $n\geq 1$, then
\begin{equation*}
T(\limsup_{n\to \infty} f_n) \geq \limsup_{n\to \infty} Tf_n.
\end{equation*}
\end{cor}

%%%%%%%%%%%%%%%%%%%%%%%%%%%%%%%%%%%%%%%%%%

\newsection{Sequences of Band projections}
We present two representation lemmas for sequences of band projections in Lemmas \ref{p5} and \ref{l3}. The first of these will
be used in the characterization of the difference between the superior and inferior limits of a sequence of band projections in Theorem \ref{l2} while the latter will be used in the proof of the Balakrishnan-Stepanov Theorem, Theorem \ref{t8}.

\begin{lem}\label{p5}
Let $E$ be a Dedekind complete Riesz space. If $(P_n)$ is a sequence of band projections on $E$, then
\begin{equation}\label{EE1}
P_n = Q+\sum_{i=1}^m Q_i,
\end{equation}
where the band projections  $Q=\prod\limits_{k=n}^{n+m}P_k$ and $Q_i=\prod\limits_{k=0}^{i-1} P_{n+k}(I-P_{n+i})$ have $QQ_i=0=Q_iQ_j$ for $i\ne j$.
\end{lem}

\begin{proof}
We observe that
\begin{equation*}\label{EE2}
 \sum\limits _{i=1}^{m}\prod\limits_{k=0}^{i-1} P_{n+k}(I-P_{n+i}) = \sum\limits_{i=1}^m\bigg(\prod\limits_{k=0}^{i-1} P_{n+k} - \prod\limits_{k=0}^i P_{n+k}\bigg)
 = P_n -\prod\limits_{k=n}^{n+m} P_k
\end{equation*}
hence proving \ref{EE1}.
Now, for $1\le i\le m$, we have
$$0\le QQ_i=\prod\limits_{k=n}^{n+m}P_k\prod\limits_{r=0}^{i-1} P_{n+r}(I-P_{n+i})
\le P_{n+i}\prod\limits_{r=0}^{i-1} P_{n+r}(I-P_{n+i})=0$$ 
and 
$$0\le Q_jQ_i=\prod\limits_{r=0}^{j-1} P_{n+r}(I-P_{n+j})\prod\limits_{k=0}^{i-1} P_{n+k}(I-P_{n+i})
\le P_{n+i}(I-P_{n+i})=0,$$ 
for $1\le i<j\le m$.
\end{proof}

\begin{lem}\label{l3}
	Let $E$ be a Dedekind complete Riesz space with weak order unit $e$. If $(P_n)$ is a sequence of band projections acting on $E$ then,
	\begin{equation}\label{m5}
	\bigvee_{n= k}^N P_n = P_k\vee\bigvee_{n= k+1}^N (P_n\prod\limits_{j=k}^{n-1}(I-P_j)).
	\end{equation}
\end{lem}
\begin{proof}
	We observe the notational convention 
	$\displaystyle{\prod\limits_{j=k}^{k-1}(I-P_j)=I}$.
Setting
   $Q_n:=P_n\prod\limits_{j=k}^{n-1}(I-P_j)=\prod\limits_{j=k}^{n-1}(I-P_j)-\prod\limits_{j=k}^{n}(I-P_j)$
 we have that $Q_nQ_m=0$ for $n\ne m$ and thus 
\begin{equation}\label{baw-2019-12}
\bigvee_{n= k}^N P_n\prod\limits_{j=k}^{n-1}(I-P_j)=\sum_{n= k}^N\left(\prod\limits_{j=k}^{n-1}(I-P_j)-\prod\limits_{j=k}^{n}(I-P_j)\right)=I-\prod\limits_{j=k}^{N}(I-P_j).
\end{equation}
Here 
\begin{equation}\label{baw-2019-13}
\prod\limits_{j=k}^{N}(I-P_j)=\bigwedge_{j=k}^{N}(I-P_j)=I-\bigvee_{j=k}^{N}P_j.
\end{equation}
Combining (\ref{baw-2019-12}) and (\ref{baw-2019-13}) gives the required equality.
\end{proof}

\begin{thm}\label{l2}
Let $E$ be a Dedekind complete Riesz space. For $(P_n)$ a sequence of band projections in $E$,
\begin{equation}\label{91}
\limsup\limits_{n\to \infty} P_n - \liminf\limits_{n\to \infty} P_n = \limsup\limits_{n\to \infty} P_n(I-P_{n+1}).
\end{equation}
\end{thm}

\begin{proof}
Since $P_n(I-P_{n+1})\le P_n$ and $P_n(I-P_{n+1})\le I-P_{n+1}$, we have
\begin{equation*}
\limsup\limits_{n\to \infty} P_n(I-P_{n+1})\le \limsup\limits_{n\to \infty} P_n
\end{equation*}
and 
\begin{equation*}
\limsup\limits_{n\to \infty} P_n(I-P_{n+1})\le I-\liminf\limits_{n\to \infty} P_n.
\end{equation*}
Thus, since $\liminf\limits_{n\to \infty} P_n\le \limsup\limits_{n\to \infty} P_n$, 
\begin{equation}\label{inequality-baw-1}
\limsup\limits_{n\to \infty} P_n(I-P_{n+1})\le \limsup\limits_{n\to \infty} P_n \wedge (I-\liminf\limits_{n\to \infty} P_n)=
\limsup\limits_{n\to \infty} P_n-\liminf\limits_{n\to \infty} P_n.
\end{equation}

Taking the limit as $m\to\infty$ in Lemma \ref{p5}, we have
\begin{equation*}
P_n = \bigwedge_{k\ge n}P_k+\bigvee_{i\ge 1} \bigwedge_{k=0}^{i-1} P_{n+k}(I-P_{n+i})
\le \bigwedge_{k\ge n}P_k+\bigvee_{i\ge 1} P_{n+i-1}(I-P_{n+i}).
\end{equation*}
Hence,
\begin{equation}\label{baw-2019-2}
P_n \le \bigwedge_{k\ge n}P_k+\bigvee_{i\ge n} P_{i}(I-P_{i+1}).
\end{equation}
Taking the $\displaystyle{\limsup_{n\to\infty}}$ of both sides to (\ref{baw-2019-2}) gives
\begin{eqnarray*}
\limsup_{n\to\infty}P_n &\le& \bigwedge_{r\ge 1}\bigvee_{n\ge r}\bigwedge_{k\ge n}P_k+\bigwedge_{r\ge 1}\bigvee_{n\ge r}\bigvee_{i\ge n} P_{i}(I-P_{i+1})\\
&=& \bigwedge_{r\ge 1}\liminf_{k\to\infty} P_k+\bigwedge_{r\ge 1}\bigvee_{i\ge r} P_{i}(I-P_{i+1})\\
&=& \liminf_{k\to\infty} P_k+\limsup_{i\to\infty} P_{i}(I-P_{i+1}).
\end{eqnarray*}
Since $\displaystyle{\limsup_{n\to\infty}P_n\cdot \liminf_{k\to\infty} P_k=\liminf_{k\to\infty} P_k}$, applying $\displaystyle{I-\liminf_{n\to\infty}P_n}$ to both sides of the above gives
\begin{equation}\label{inequality-baw-2}
\limsup_{n\to\infty}P_n -\liminf_{k\to\infty} P_k\le 
\limsup_{i\to\infty} P_{i}(I-P_{i+1}).
\end{equation}
Combining (\ref{inequality-baw-1}) and (\ref{inequality-baw-2}) concludes the proof.
\end{proof}
%%%%%%%%%%%%%%%%%%%%%%%%%%%%%%%%%%%%%%%%%%

\newsection{Borel-Cantelli type theorems in Riesz spaces}

In \cite[Corollary 4.2]{wen3} a generalization of the First Borel-Cantelli Lemma to Riesz spaces was given and it is quoted below.
Following this we generalize the Barndorff-Nielsen Theorem, in
Theorem \ref{t7}, and the Balakrishnan-Stepanov Theorem, in Theorem \ref{t8}, to Dedekind complete Riesz spaces with weak order unit and conditional expectation operator.

\begin{cor}[Borel-Cantelli Theorem in Riesz spaces]\label{c1}
Let $E$ be a Dedekind complete Riesz space with strictly positive conditional expectation operator T and weak order unit $e = Te$. 
Let be $(P_j)$ be a sequence of band projections in $E$.	
If $\displaystyle{\sum_{j=1}^{\infty}T(P_je)}$ is order convergent in $E$,  then $\displaystyle{\limsup\limits_{j\to \infty} P_j = 0}.$
\end{cor}

\begin{thm}[Barndorff-Nielsen Theorem in Riesz spaces]\label{t7}
Let $(P_n)$ be a sequence of band projections on a Dedekind complete Riesz space $E$ with strictly positive conditional expectation operator $T$ and weak order unit $e = Te$ such that 
$\displaystyle{\liminf_{n\to \infty} T(P_ne)= 0}$ and
$\displaystyle{\sum\limits_{n=1}^\infty T\big(P_n(I-P_{n+1})e\big)}$ is convergent in $E$,
then
$\displaystyle{\limsup_{n\to \infty} P_n= 0}.$
\end{thm}

\begin{proof}
Let
\begin{equation}\label{q6.1}
Q_n = P_n(I-P_{n+1}) , n\geq 1.
\end{equation}
Applying Corollary \ref{c1} to $(Q_n)$, as $\sum\limits_{n=1}^\infty  TQ_ne$ converges in $E$ then $\underset{n \to \infty } {\limsup}\ Q_n = 0$. 
Thus, $\underset{n \to \infty } {\limsup}\  P_n(I-P_{n+1})= 0$, so, by Theorem \ref{l1}, $$\underset{n \to \infty } {\liminf}\ P_n = \underset{n \to \infty } {\limsup}\ P_n= \underset{n \to \infty } {\lim}\ P_n.$$
Hence, as $T$ is order continuous,
\begin{equation*}
0=\liminf_{n\to \infty} T(P_ne)=\underset{n \to \infty } {\lim}\ TP_ne=T(\underset{n \to \infty } {\lim}\ P_ne)=T( \liminf_{n\to \infty}  P_ne).
\end{equation*}
Now, since $T$ is a strictly positive,  we have that
$\displaystyle{\liminf_{n\to \infty} P_ne = 0}.$
\end{proof}

\begin{thm}[Balakrishnan-Stepanov Theorem in Riesz spaces]\label{t8}
Let $E$ be a Dedekind complete Riesz space with  $T$ a strictly positive conditional expectation operator on $E$ and weak order unit $e=Te$. Let $(P_n)$ be a sequence of band projections  on $E$ with  $T(P_ne) \to 0$ in order, as $n\to\infty$, and 
\begin{equation}\label{BS-assumption}
 \sum\limits_{n=1}^\infty T\Bigg(P_{n+m}\prod_{j=n}^{n+m-1}(I-P_j)e\Bigg)
\end{equation}
order convergent in $E$ for some $m \geq 1$. Then $\underset{n \to \infty}{\lim \sup}\ P_n  = 0$.
\end{thm}

\begin{proof}
From definition of the limit supremum we have for each $k\in\N$,
\begin{equation}\label{2019-sup}
T\bigg(\limsup_{n\to \infty}P_ne \bigg) = T\Bigg(\bigwedge_{j\in \mathbb{N}}\bigvee_{n\geq j}P_ne\Bigg)
\leq T\Bigg(\bigvee_{n\geq k}P_ne\Bigg). 
\end{equation}
Letting $N\to \infty$ in Lemma \ref{l3}, we get
\begin{equation}\label{eq5.w}
\bigvee_{n\geq k}P_ne= \bigvee_{n\geq k}P_n\prod_{j=k}^{n-1}(I-P_j)e.
\end{equation}
Combining (\ref{2019-sup}) and (\ref{eq5.w}), we obtain
\begin{align}\label{eq5}
T\bigg(\limsup_{n\to \infty}P_ne \bigg)\leq T\Bigg(\bigvee_{n\geq k}P_ne\Bigg)=T\Bigg(\bigvee_{n\geq k}P_n\prod_{j=k}^{n-1}(I-P_j)e\Bigg).
\end{align}
However, as shown in Lemma \ref{l3}, for $n\ne m$,
$$(P_n\prod_{j=k}^{n-1}(I-P_j))(P_m\prod_{j=k}^{m-1}(I-P_j))=0,$$ thus,  
\begin{equation}\label{2019-new-77}
\bigvee_{n\geq k}P_n\prod_{j=k}^{n-1}(I-P_j)=\sum_{n\geq k}P_n\prod_{j=k}^{n-1}(I-P_j).
\end{equation}
Combining (\ref{eq5}) and (\ref{2019-new-77}) gives, for $k\in\N$,
\begin{align}\label{eq51}
T\bigg(\limsup_{n\to \infty}P_ne \bigg)\leq \sum_{n\geq k}TP_n\prod_{j=k}^{n-1}(I-P_j)e.
\end{align}
Now for $k\ge m+1$ we have
\begin{align}\label{eq511}
\sum_{n\geq k}TP_n\prod_{j=k}^{n-1}(I-P_j)e=\sum_{n= k}^{k+m-1}TP_n\prod_{j=k}^{n-1}(I-P_j)e
+\sum_{n\ge k}TP_{n+m}\prod_{j=k}^{n+m-1}(I-P_j)e.
\end{align}
For $n\ge k$ we have $\displaystyle{\prod_{j=k}^{n+m-1}(I-P_j)\le\prod_{j=n}^{n+m-1}(I-P_j)}$ giving
\begin{align}\label{eq512}
\sum_{n\ge k}TP_{n+m}\prod_{j=k}^{n+m-1}(I-P_j)e\le \sum_{n\ge k}TP_{n+m}\prod_{j=n}^{n+m-1}(I-P_j)
\end{align}
which tends to zero in order as $k\to\infty$ since  (\ref{BS-assumption}) is order convergent in $E$. 
Further
\begin{align}\label{eq513}
\sum_{n= k}^{k+m-1}TP_n\prod_{j=k}^{n-1}(I-P_j)e\le \sum_{n= k}^{k+m-1}(TP_ne)
\end{align}
which tends to zero in $E$ as $k\to\infty$ as $TP_ne\to 0$ in order as $n\to \infty$.
\end{proof}

We note that (\ref{BS-assumption}) and (\ref{eq5}) with Theorem 
\ref{l2} give that $\displaystyle{\liminf_{n\to\infty}P_n}=0$. 
Thus the assumption $\displaystyle{\lim_{n\to\infty}TP_ne}=0$, as used in Theorem \ref{t8}, cannot be replaced by $\displaystyle{\liminf_{n\to\infty}TP_ne}=0$.

%%%%%%%%%%%%%%%%%%%%%%%%%%%%%%%%%%%%%%

%%%%%%%%%%%%%%%%%%%%%%

\begin{thebibliography}{99}
\addcontentsline{toc}{chapter}{Bibliography}
%
\bibitem{azouzi}{{\sc Y. Azouzi, K. Ramdane}, {Burkholder Inequalities in Riesz spaces}, {\em Indag. Math.}, {\bf 28} {(2017), 1076-1094.}}	
\bibitem{ali} {\sc C. D. Aliprantis, K. C. Border,} {{\em Infinite dimensional analysis: A hitchhiker's guide.} {Springer-Verlag,}} (1994).
\bibitem{bs} {\sc N. Balakrishnan, A. Stepanov,} {{A generalization of the Borel-Cantelli Lemma}.  {\em Math. Scientist}, {\bf 35} (2010), 61-61.}
\bibitem{bn} {\sc O. Barndorff-Nielsen,} {{On the rate of growth of the maxima of a sequence of independent identically distributed random variables}.  {\em Math. Scand.}, {\bf 9} (1961), 383-394.}
\bibitem{chan}{\sc T.K. Chandra,} {{\em The Borel-Cantelli Lemma.}, {Springer}, (2012)}.
\bibitem{grobler}{\sc J.J. Grobler, C.C.A. Labuschagne,} {{Girsanov’s theorem in vector lattices}. {\em Positivity}, to appear, DOI: https://doi.org/10.1007/S11117-019-00649-5} 
\bibitem{wen1}{\sc W.-C. Kuo, C.C.A. Labuschagne, B.A. Watson,} {{Discrete time stochastic processes in Riesz spaces}. {\em Indag. Math,}{\bf 15} (2004), 15-40.}  
\bibitem{wen2}{\sc W.-C. Kuo, C.C.A. Labuschagne, B.A. Watson,} {{Conditional Expectation on Riesz spaces}. {\em J. Math. Anal. Appl,} {\bf 303} (2005), 509-521.}
\bibitem{wen3}{{\sc W.-C. Kuo, C.C.A. Labuschagne, B.A. Watson},
         {Zero-one laws for Riesz space and fuzzy random variables},
         {\em Fuzzy Logic, Soft Computing and Computational Intelligence}, Springer Verlag and 
         Tsinghua University Press, Beijing, China, 2005, pages 393-397.}
\bibitem{lux}{\sc W.A.J. Luxemburg, A.C. Zaanen,} {{\em Riesz spaces I}. North Holland, (1971).}
 \bibitem{M-N}{{\sc P. Meyer-Nieberg},
         {\em Banach Lattices},
          Springer Verlag, 1991.}
\bibitem{za1}{\sc A.C. Zaanen,} {{\em Introduction to Operator Theory in Riesz Spaces}. {Springer-Verlag}, (1997).}
\end{thebibliography}
\end{document}